\documentclass[11pt, final]{amsart}
\usepackage[english]{babel}
\usepackage{latexsym}
\usepackage{amssymb}
\usepackage{amscd}
\usepackage[all]{xy}
\usepackage[cp850]{inputenc}
\usepackage[mathscr]{eucal}

\theoremstyle{plain}

\newtheorem{theorem}{Theorem}[section] 
\newtheorem{lemma}[theorem]{Lemma}     
\newtheorem{corollary}[theorem]{Corollary}
\newtheorem{proposition}[theorem]{Proposition}

\newcommand{\adef}{\begin{defn}}
\newcommand{\zdef}{\end{defn}}
 \newtheorem{defn}[theorem]{Definition}

\theoremstyle{definition}

\newcommand\supp{\mathop{\mathrm{supp}}\nolimits}
\newcommand{\lop}{\ensuremath{\curvearrowright}}
\theoremstyle{remark}

\newcommand{\sn}{\sum_{k=1}^n}

\newcommand{\R}{\mathbb{R}}
\newcommand{\C}{\mathbb{C}}
\newcommand{\U}{\mathscr{U}}
\newcommand{\N}{\mathbb{N}}

\newcommand{\aproof}{\begin{proof}}
\newcommand{\zproof}{\end{proof}}

\newcommand{\To}{\longrightarrow}

\title{On Disjointly singular centralizers}
\author{Jes\'us M. F. Castillo}
\address{Instituto de Matem\'aticas, Universidad de Extremadura,
Avenida de Elvas s/n, 06011 Badajoz, Spain.} \email{castillo@unex.es}

\author{Wilson Cuellar}
\address{Departamento de Matem\'atica, Instituto de Matem\'atica e
Estat\'\i stica, Universidade de S\~ao Paulo, rua do Mat\~ao 1010,
05508-090 S\~ao Paulo SP, Brazil} \email{cuellar@ime.usp.br}

\author{Valentin Ferenczi}
\address{Departamento de Matem\'atica, Instituto de Matem\'atica e
Estat\'\i stica, Universidade de S\~ao Paulo, rua do Mat\~ao 1010,
05508-090 S\~ao Paulo SP, Brazil  \\ and Equipe d'Analyse Fonctionnelle \\
Institut de Math\'ematiques de Jussieu \\
Universit\'e Pierre et Marie Curie - Paris 6 \\
Case 247, 4 place Jussieu \\
75252 Paris Cedex 05 \\
France.} \email{ferenczi@ime.usp.br}

\author{Yolanda Moreno}
\address{Instituto de Matem\'aticas, Universidad de Extremadura,
Avenida de Elvas s/n, 06011 Badajoz, Spain.} \email{ymoreno@unex.es}


\thanks{The research of the first and fourth authors has been supported in part by project MTM2016-76958-C2-1-P de MINCIN and Project IB16056 de la Junta de Extremadura, Spain.  The second  author was supported by Fapesp, grants 2016/25574-8 and 2018/18593-1. The third author was supported by CNPq grant 303034/2015-7, and by Fapesp, grants 2015/17216-1 and 2016/25574-8.
}

\begin{document}


\maketitle

\begin{abstract} We study ``disjoint" versions of the notions of trivial, locally trivial, strictly singular and super-strictly singular quasi-linear maps in the context of K\"othe function spaces. Among other results, we show: i) (locally) trivial and (locally) disjointly trivial notions coincide on reflexive spaces; ii) On non-atomic superreflexive K\"othe spaces, no centralizer is singular, although most are disjointly singular.
iii) No super singular quasi-linear maps exist between superreflexive spaces although Kalton-Peck centralizers are super disjointly singular; iv) Disjoint singularity does not imply super disjoint singularity.
\end{abstract}

\section{Introduction}

For all unexplained notation and terms, please keep reading. This paper has its roots in \cite{cfg} where the authors introduced the notion of disjointly singular centralizer on K\"othe function spaces, proved that disjoint singularity coincides with singularity on Banach spaces with unconditional basis and presented a technique to produce disjointly singular centralizers via complex interpolation.

A second equally important fact to consider is that the fundamental  Kalton-Peck map \cite{kaltpeck} is disjointly singular on $L_p$ \cite[Proposition 5.4]{cfg}, but it is not singular \cite{suakp}. In fact, as the last stroke one could wish to foster the study of disjoint singularity is the argument of Cabello \cite{cabekp} that no centralizer on $L_p$ can be singular that we extend here by showing that no centralizer can be singular. It is thus obvious that while singularity is an important notion in the domain of K\"othe sequence spaces, disjoint singularity is the core notion in K\"othe function spaces. The purpose of this paper is then to study the \emph{disjointly supported} versions of the basic (trivial, locally trivial, singular and supersingular) notions in the theory of centralizers and present several crucial examples.

\section{Background}
 Most of the action in this paper will take place in the ambient of K\"othe  functions spaces over a $\sigma$-finite measure space $(\Sigma, \mu)$ endowed with their $L_\infty$-module structure. A K\"othe function space $K$ is a linear subspace of $L_0(\Sigma, \mu)$, the vector space of all measurable functions, endowed with a quasi-norm such that whenever $|f|\leq g $ and $g\in K$ then $f\in K$ and $\|f\|\leq \|g\|$ and so that for every finite measure subset $A\subset \Sigma$ the characteristic function $1_A$ belongs to $X$. A particular case of which is that of Banach spaces with a 1-unconditional basis (called K\"othe sequence spaces in what follows) with their associated $\ell_\infty$-module structure.

\subsection{Exact sequences, quasi-linear maps and centralizers}

For a rather complete background on the theory of twisted sums see \cite{castgonz}. We recall that  a twisted sum of two Banach spaces $Y$,  $Z$ is  a quasi-Banach space $X$ which has a closed subspace isomorphic to $Y$ such that the quotient $X/Y$ is isomorphic to $Z$. Equivalently, $X$ is a twisted sum of $Y$, $Z$ if there exists a short exact sequence
$$\begin{CD}
0@>>>  Y@>>> Z @>>> X@>>> 0.\end{CD}$$

According to Kalton  and Peck \cite{kaltpeck}, twisted sums  can be identified with homogeneous maps $\Omega: X \to Y$ satisfying
\[ \| \Omega (x_1+x_2) - \Omega x_1- \Omega x_2\| \leq C(\|x_1\|+ \|x_2\|),\]
which are called  quasi-linear maps, and induce an equivalent quasi-norm on $X$ (seen
algebraically as $Y \times X$) by $$\|(y,x)\|_\Omega=\|y-\Omega z\| +\|x\|.$$
This space is usually denoted $Y\oplus_\Omega X$. When $Y$ and $X$ are, for example,  Banach spaces  of non-trivial type, the quasi-norm above is equivalent to a norm;  therefore, the twisted sum obtained is a Banach space. Two exact sequences $0 \to Y \to Z_1 \to X \to 0$ and $0 \to Y \to Z_2 \to X \to 0$
are said to be {\it equivalent} if there exists an operator $T:Z_1\to Z_2$ such that the following
diagram commutes:
$$
\begin{CD}
0 @>>>Y@>>>Z_1@>>>X@>>>0\\
&&@| @VVTV @|\\
0 @>>>Y@>>>Z_2@>>>Z@>>>0.
\end{CD}$$
The classical 3-lemma (see \cite[p. 3]{castgonz}) shows that $T$ must be an isomorphism.

\adef An $L_\infty$-centralizer (resp. an $\ell_\infty$-centralizer) on a K\"othe function (resp. sequence) space $\mathcal K$  is a homogeneous map $\Omega: \mathcal K \to L_0$  such that  there is a constant $C$  satisfying that,  for every $f\in L_\infty$ (resp. $\ell_\infty$) and for every $x\in \mathcal K$, the difference $\Omega(fx)- f\Omega(x)$ belongs to $\mathcal K$ and
$$ \| \Omega(fx)- f\Omega(x)\|_{\mathcal K}\leq C\|f\|_\infty  \|x\|_{\mathcal K}. $$
The centralizer is called real when it sends real functions (sequences) to real functions (sequences).
\zdef
When no confusion arises we will simply say: a centralizer. Observe that a centralizer $\Omega$ on $\mathcal K$ does not take values in $\mathcal K$, but in $L_0$, and still it induces an exact sequence
$$\begin{CD}
0@>>>  \mathcal K @>\jmath >> d_\Omega \mathcal K @>Q>> \mathcal K@>>> 0\end{CD}$$
as follows: $d_\Omega \mathcal K= \{ (w, x) : w\in L_0, x\in \mathcal K: w - \Omega x\in \mathcal K\}$ endowed with the
quasi-norm $$\|(w,x)\|_{d_\Omega \mathcal K}=\|x\|_{\mathcal K} +\|w- \Omega x\|_{\mathcal K}$$ and with obvious inclusion $\jmath(x) = (x, 0)$ and quotient map $Q(w,x)=x$.
The reason is that a centralizer ``is" quasi-linear, in the sense that for all $x,y\in \mathcal K$ one has $\Omega(x+y) - \Omega(x) - \Omega(y) \in \mathcal K$ and $\|\Omega(x+y) - \Omega(x) - \Omega(y)\|\leq C(\|x\|+\|y\|)$ for some $C>0$ and all $x,y\in \mathcal K$.
To describe the fact that the centralizer acts $\Omega: \mathcal K\to L_0$ but defines a twisted sum of $\mathcal K$ with itself we will use sometimes the notation $\Omega: \mathcal K \lop \mathcal K$. Centralizers arise naturally by complex interpolation \cite{Bergh-Lofstrom} as can be seen in \cite{kaltdiff}.

 \subsection{Trivial maps}
An exact sequence $0 \to Y \to Z \to X \to 0$ is trivial if and only if it is equivalent to
$0 \to Y \to Y \oplus X \to X \to 0$, where $Y\oplus X$ is endowed with the product norm.
In this case we say that the exact sequence \emph{splits.} Two quasi-linear maps $\Omega, \Omega': X \to Y$ are said to be equivalent, denoted $\Omega\equiv \Omega'$,
if the difference $\Omega-\Omega'$ can be written as $B +L$, where $B: X \to Y$ is a homogeneous bounded
map (not necessarily linear) and $L: X \to Y$ is a linear map (not necessarily bounded).
Two quasi-linear maps are equivalent if and only if the associated exact sequences
are equivalent. A quasi-linear map is trivial if it is equivalent to the $0$ map,  which also means that  the associated exact sequence is trivial. Given two Banach spaces $X,Y$ we will denote by $\ell(X,Y)$ the vector space of linear (not necessarily continuous) maps $X\to Y$. The distance between two homogeneous maps $T,S$ will be the usual operator norm (the supremum on the unit ball) of the difference; i.e., $\|T-S\|$, which can make sense even when $S$ and $T$ are unbounded. So a quasi-linear map $\Omega: X \To Y$ is trivial if and only if $d(\Omega,\ell(X,Y))\leq C <+\infty$, in which case we will say that $\Omega$ is $C$-trivial. A centralizer $\mathcal K\lop \mathcal K$ is trivial if and only if there is a linear map $L: \mathcal K\to L_0$ so that $\Omega - L: \mathcal K\to \mathcal K$ is bounded.

 \subsection{Locally trivial maps}

 A quasi-linear map $\Omega: X \to  Y$ is said to be {\em locally trivial} \cite{kaltloc} if there exists $C>0$ such that or any finite dimensional subspace $F$ of $X$, there exists a linear map $L_F$ such that $\|\Omega_{|F}-L_F\| \leq C$. It is clear that a trivial map is locally trivial. The converse is not true, although locally trivial quasi-linear maps $F: X\to Y$ in which $Y$ is reflexive are trivial, by \cite{cabecastuni}.

 \subsection{Singular maps}

An operator between Banach spaces is said to be \emph{strictly singular}
if no restriction to an infinite dimensional closed subspace is an isomorphism. Analogously,
a quasi-linear map (in particular, a centralizer) is said to be \emph{singular} if its
restriction to every infinite dimensional closed subspace is never trivial.
An exact sequence induced by a singular quasi-linear map is called a \emph{singular sequence.} A  quasi-linear map is singular if and only
if the associated exact sequence has strictly singular quotient map  \cite[Lemma 1]{castmorestrict}. Singular $\ell_\infty$-centralizers exist and the most natural example is the Kalton-Peck map $\mathscr K_p: \ell_p\lop \ell_p$, $0<p<+\infty$, defined by $\mathscr K_p(x)  = x\log \frac{|x|}{\|x\|_p}.$ The proof that $\mathscr K_p$ is singular can be found in \cite{kaltpeck} for $1<p<+\infty$, \cite{castmorestrict} for $p=1$, and \cite{ccs} for all $0<p<+\infty$. A simple characterization of singular $\ell_\infty$-centralizers on Banach sequence spaces can be presented

\begin{proposition}\label{sin} Let $X$ be a Banach space with an unconditional basis not containing $c_0$. Let $\Omega:X \lop X$ be an $\ell_\infty$-centralizer such that for every sequence $(A_k)$ of
consecutive subsets of $\N$ and every sequence $(u_n)$ of consecutive normalized blocks of
the basis, for which $\sup_k \|\sum_{n\in A_k} u_n\| \to +\infty$ one has
\begin{equation*}\lim \sup_k \frac{\|\Omega(\sum_{n\in A_k} u_n) -
\sum_{n\in A_k} \Omega( u_n) \|}{\|\sum_{n\in A_k} u_n\|}
=+\infty.\end{equation*}
Then $\Omega$ is  singular.
\end{proposition}
\begin{proof} If $\Omega: X\lop X$ is an $\ell_\infty$-centralizer verifying the condition above and, at the same time, trivial on some subspace $H$, by the blocking principle
(see \cite{ccs}), it must be trivial on the subspace $[u_n]$ spanned by some consecutive blocks
of the basis. Standard manipulations (see \cite{ccs,cfg}) show that the linear map $\ell(u_n)=\Omega(u_n)$ is at finite distance from $\Omega$, which implies that $\lim \sup_k \|\sum_{n\in A_k} u_n\| <+\infty$ for all choices of $(A_k)$, thus $(u_n)$ is equivalent to the canonical basis of $c_0$ and consequently $H$ contains $c_0$. \end{proof}

In sharp contrast, Cabello \cite{cabekp} proved that no $L_\infty$-centralizer is singular on $L_p[0,1]$. Let us observe that quite the same proof of Cabello provides:

\begin{proposition}\label{superfelix} No singular $L_\infty$-centralizers exist on (admissible) superreflexive K\"othe funcion spaces. More precisely, every $L_\infty$-centralizer on an admissible superreflexive K\"othe function space is bounded on some copy of $\ell_2$.\end{proposition}

\emph{Sketch of proof}: Recall that according to Kalton \cite[p.482]{kaltdiff} a K\"othe space is termed admissible when for some strictly positive functions $h, k\in L_0$ one has $\|hk\|_1\leq \|x\|_{\mathcal K}\leq \|kx\|_\infty$ for every $x\in \mathcal K$. By Kalton's theorem \cite[Thm. 7.6]{kaltdiff} plus the comments in \cite[Section 1.3]{cabekp} there are two admissible K\"othe spaces $A,B$ so that $\mathcal K=(A,B)_{1/2}$; these spaces can be assumed to be superreflexive by reiteration and \cite[Thm. 7.8]{kalt-mon}. The admissibility hypothesis yields functions $h_a, k_a$ such that $\|h_a f\| \leq \|f\|_A\leq \|k_a f\|_\infty$ for every $f \in A$; and functions $h_b, k_b$ such that
$\|h_b f\| \leq \|f\|_B\leq \|k_b f\|_\infty$ for every $f \in B$. Thus, one can find a positive measure set $S\subset \Sigma$ and a constant $M>0$ such that $k_a, k_b \leq M$ and $h_a, h_b \geq M^{-1}$. This provides continuous inclusions $L_\infty(S)\subset A(S)\subset L_1(S)$ and $L_\infty(S)\subset B(S)\subset L_1(S)$ .

By super-reflexivity, both spaces $A,B$ are $p$-convex and $q$-concave for some $1<p,q<+\infty$ (\cite[Thm 1.f.12 and  Thm 1.f.7.]{lindtzaf2})
So, using the Johnson-Maurey-Schechtman-Tzafriri remark \cite[p.14]{jmst} then also $L_q(S)\subset A(S)\subset L_p(S)$ and $L_q(S)\subset B(S)\subset L_p(S)$ . Let $R(S)$ be the subspace generated by Rademacher functions supported in $S$. The $L_p$ and $L_q$-norms are equivalent on $R(S)$ by Khintchine's inequality, and are also equivalent to $\|\cdot\|_A$ and to $\|\cdot\|_B$, and thus $R(S)\sim \ell_2$. The equivalence of norms $A$ and $B$ on $R(S)$ makes the differential $\Omega_{1/2}$ bounded on $R(S)$, and since $\Omega_K$ is boundedly equivalent to $\Omega_{1/2}$, it must be bounded too.

\subsection{Super-singular maps}

An operator $T: X\to Y$ between two Banach spaces is said to be {\em super strictly singular} (in short, super-SS) if  there does not exist a number $c>0$ and a sequence of subspaces $E_n$ of $X$, with $\dim \, E_n=n$, such that $\|Tx\|\geq c\|x\|$ for every $x\in \bigcup_n E_n$.
Equivalently \cite[Lemma 1.1.]{css}, if every ultrapower of $T$ is strictly singular. Super strictly singular operators have also been called finitely strictly singular; they were first introduced in \cite{M, M2}, and form a closed ideal containing the ideal of compact and contained in the ideal of strictly singular operators. See also \cite{css} for the study of such a notion in the context of twisted sums, as well as \cite{ccfm} where a few results are also mentioned in relation to complex structures on twisted sums.

It is a standard fact (see \cite{castgonz}) that given an exact sequence $0\to Y\to Z \to Z \to 0$ and an ultrafilter $\U$ the ultrapowers form an exact sequence $0\to Y_\U\to Z_\U \to X_\U \to 0$. If $\Omega$ is a quasi-linear map associated to the former sequence we will call $\Omega_\U$ any quasi-linear map associated to the later. We do not need for the moment to specify the construction of $\Omega_\U$.
We will say, following \cite{ccfm} that a quasi-linear map $\Omega$ is super-singular if every ultrapower $\Omega_\U$ is singular. We need to state here two facts proved in \cite{ccfm}:
\begin{itemize}
\item  $\Omega$ is super-singular if and only the  quotient map  $q_\Omega$ of the exact sequence it defines is super strictly singular.
\item No super singular quasi-linear maps between $B$-convex Banach spaces exist. This follows from \cite[Thm. 3]{pliv}, where it is proved that a super strictly singular operator on a $B$-convex space has super strictly singular adjoint. Since superreflexive spaces are $B$-convex, $B$-convexity is a $3$-space property (see \cite{castgonz}) and  the adjoint of a quotient map is an into isomorphism, the result follows.\end{itemize}

After these prolegomena, we tackle the study of the ``disjoint" versions of the preceding properties. It is worth to observe that all our forthcoming ``disjoint" notions admits an immediate translation to general quasi-linear maps on Banach lattices.

\section{Disjoint local triviality} \adef A quasi-linear map $\Omega: \mathcal K \to  Y$ defined on a Banach lattice is said to be {\em disjointly trivial} if it is trivial on any subspace generated by a sequence of disjointly supported elements. It is said to be {\em locally disjointly trivial} if there exists $C>0$ such that for any finite dimensional subspace $F$ of $\mathcal K$ generated by disjointly supported vectors, there exists a linear map $L_F$ such that $\|\Omega_{|F}-L_F\| \leq C$.\zdef

We can show:

\begin{proposition}\label{bounded}
Let $\Omega: \mathcal K \to Y$ be a quasi-linear map on a K\"othe space $\mathcal K$. Consider the following assertions:
\begin{itemize}
\item[(i)] $\Omega$ is trivial.
\item[(ii)] $\Omega$ is disjointly trivial.
\item[(iii)] $\Omega$ is locally disjointly trivial.
\item[(iv)] $\Omega$ is locally trivial.
\end{itemize}
Then $(i) \Rightarrow (ii) \Rightarrow (iii) \Leftrightarrow (iv)$. Moreover, if $Y$ is complemented in its bidual, then all assertions are equivalent.
\end{proposition}
\begin{proof}
Assertions $(i)$ and $(iv)$ are well-known to be equivalent: trivial implies locally trivial while, see \cite{cabecastuni}, a locally trivial quasi-linear map taking values in a space complemented in its bidual is trivial. That $(i)$ implies $(ii)$ is obvious. Let us show that $(iii)$ implies $(iv)$: Let $\Omega$ be a quasi-linear map verifying $(iii)$ and let $F$ be a finite dimensional subspace of $\mathcal K$. Approximating functions by characteristic functions we may find a nuclear operator $N$ on $\mathcal K$ of arbitrary small norm so that $(Id+N)(F)$ is contained in the linear span $ [u_n]$ of a finite sequence of disjointly supported vectors. The restriction $\Omega_{|[u_n]}$ is trivial with constant $C$, thus using \cite[Lemma 5.6]{ccfm}, we get that $\Omega=\Omega(I+N)-\Omega N$ is trivial with constant $C+\epsilon$ on $F$.
Therefore $(iv)$ holds.\medskip

It remains to show that $(ii)$ implies $(iii)$. Let $\Sigma$ be the $\sigma$-finite base space on which the K\"othe function space $\mathcal K$ is defined.  For a subset $A\subset \Sigma$ we will denote $\mathcal K(A)$ the subspace of $\mathcal K$ formed by those functions with support contained in $A$.\medskip

\textbf{Claim 1.} \emph{If $A$ and $B$ are disjoint and $\Omega$ is trivial on both $\mathcal K(A)$ and $\mathcal K(B)$ then it is trivial on $\mathcal K(A \cup B)$.} Indeed, if $\|\Omega_{|\mathcal K(A)}-a\| \leq c$ and
$\|\Omega_{|\mathcal K(B)}-b\| \leq d$, where $a$ and $b$ are linear, then $\|\Omega_{|\mathcal K(A \cup B)}-(a \oplus b)\| \leq 2(Z(\Omega)+c+d)$, where $a \oplus b$ is the obvious linear map on $\mathcal K(A \cup B)$.\medskip

\textbf{Claim 2.} \emph{If $\Omega$ is nontrivial on $X$ then $\Sigma$ can be split in two sets $\Sigma = A \cup B$ so that $\Omega_{|\mathcal K(A)}$ and $\Omega_{|\mathcal K(B)}$ are both nontrivial}.  We first assume that $\Sigma$ is a finite measure space. Assume the claim does not hold. Split $\Sigma = R_1 \cup I_1$ in two sets  of the same measure and assume $\Omega_{|\mathcal K(I_1)}$ is trivial. Note that since the claim does not hold, given any  $C \subset \Sigma$ and any splitting $C=A \cup B$ the map $\Omega$ is trivial on $\mathcal K(A)$ or $\mathcal K(B)$. So, split $R_1  = R_2 \cup I_2$ in two sets of equal measure and assume that  $\Omega_{|\mathcal K(I_2)}$ is trivial, and so on. If $\Omega$ is $\lambda$-trivial on $\mathcal K(\cup_{j \leq n} I_j)$ for $\lambda<+\infty$ and for all $n$ then $\Omega$ is locally trivial on $\mathcal K$ and therefore is trivial, a contradiction. If $\lambda_n\to \infty$ is such that $\Omega_{\mathcal K(\cup_{j \leq n}I_j)}$ is $\lambda_n+1$-trivial but not $\lambda_n$-trivial for all $n$, then by the Fact we note that
for $m<n$, $\Omega$ cannot be trivial with constant less
 than $\lambda_n/2-Z(\Omega)-\lambda_m-1$ on $\mathcal K(\cup_{m<j \leq n} I_j)$. From this we find a partition of $\N$ as $N_1 \cup N_2$ so that if $A=\cup_{n \in N_1} I_n$ and
$B=\cup_{n \in N_2} I_n$, then $\Omega$ is non trivial on $\mathcal K(A)$ and $\mathcal K(B)$, another contradiction.\medskip

If $\Sigma$ is $\sigma$-finite then the proof is essentially the same: either one can choose the sets $I_n$ all having measure, say, $1$  or at some step $R_m$ is of finite measure, and we are in the previous case. This concludes the proof of the claim.\medskip

We pass to complete the proof that $(ii)$ implies $(iii)$. Assume that $\Omega$ is not trivial on $\mathcal K$. By the claim, split $\Sigma = A_1 \cup B_1$ so that  $\Omega$ is trivial neither on $\mathcal K(A_1)$ nor on $\mathcal K(B_1)$. It cannot be locally trivial on them, so there is a finite number $\{u^1_n\}_{n\in F_1}$ of disjointly supported vectors on $\mathcal K(A_1)$ on which $\Omega$ is not $2$-trivial. By the claim applied to $\mathcal K(B_1)$ split $B_1= A_2 \cup B_2$ so that $\Omega$ is  trivial  neither on $\mathcal K(A_2)$ nor in $\mathcal K(B_2)$. It cannot be locally trivial on them, so there is a finite number of disjointly supported vectors $\{u^2_n\}_{n\in F_2}$ on $X(A_2)$ on which $\Omega$ is not $4$-trivial. Iterate the argument to produce a subspace $Y$ generated by an infinite sequence
$$ \{u^1_n\}_{n\in F_1} , \{u^2_n\}_{n\in F_2}, \dots, \{u^k_n\}_{n\in F_k}, \dots$$
of disjointly supported vectors, where $\Omega$ cannot be trivial.\end{proof}

An immediate corollary of (the proof of) Proposition \ref{bounded} is:

\begin{corollary} Given a K\"othe space $\mathcal K$ with base space $(\Sigma, \mu)$ and a non-locally trivial quasi-linear map $\Omega$ defined on $\mathcal K$ then there is a sequence $(A_n)$ of finite measure mutually disjoint subsets of $\Sigma$ so that the restriction $\Omega_{|[1_{A_n}]}$ is not locally trivial.
\end{corollary}

In $L_p(\Sigma, \mu)$ a subspace $[1_{A_n}]$ is isomorphic to $\ell_p$ when $0<p<+\infty$ and to $c_0$ in $L_\infty(\Sigma, \mu)$.
It is not clear whether Proposition \ref{bounded} can be translated to the domain of Banach lattices. $C(K)$-spaces
are not, as a rule, K\"othe spaces; however, the following essential part of Proposition \ref{bounded} still survives \cite[Theorem 2.1]{ccalb} and since a disjointly supported sequence in a $C(K)$-space generates $c_0$, one has:

\begin{corollary}\label{este} Let $0<p <\infty$. Given a non-locally trivial map $\Omega$ defined on $L_p(\Sigma, \mu)$ (resp. $L_\infty(\Sigma, \mu)$, $C(K)$) there is a copy of $\ell_p$ (resp. $c_0$) spanned by disjointly supported vectors on which the restriction of $\Omega$ is not locally trivial.\end{corollary}

\section{Disjoint singularity}

Theorem \ref{bounded} shows that (local) triviality and disjoint (local) triviality are essentially equivalent. We shall now see that the situation is much more complex regarding {\em singularity} notions.

\adef A quasi-linear map on a Banach lattice is called \emph{disjointly singular} if
its restriction to every infinite dimensional subspace generated by a disjointly supported sequence is never trivial.
\zdef

Of course that a singular quasi-linear map is disjointly singular and a disjointly singular quasi-linear map on a K\"othe sequence space is singular. An open question, to the best of our knowledge due to F\'elix Cabello, is about the existence of singular quasi-linear maps on K\"othe function spaces; recall that no singular $L_\infty$-centralizers exist on any reasonable K\"othe space \cite{cabekp} (cf. Proposition \ref{superfelix}); see also \cite{suakp}).

\subsection{Examples}\label{examples}
\begin{enumerate}
\item As we mentioned at the introduction, the methods in \cite{cfg} actually produce disjointly singular centralizers. In particular, it is shown \cite[Proposition 5.4]{cfg} that the Kalton-Peck centralizer
$$\mathscr K(x)=x\log \frac{|x|}{\|x\|}$$
is disjointly singular on any reflexive, $p$-convex K\"othe function space, $p>1$

\item Given two Lorentz spaces $L_{p_0, q_0}, L_{p_1, q_1}$, it was proved in \cite{cabelo} that $(L_{p_0, q_0}, L_{p_1, q_1})_\theta = L_{p,q}$ for
$p^{-1} = (1-\theta){p_0}^{-1} + \theta{p_1}^{-1}$ and $q^{-1} = (1-\theta){q_0}^{-1} + \theta{q_1}^{-1}$
with associated derivation
$$
\Omega(x)=  q\left(\dfrac{1}{q_1}-\frac{1}{q_0}\right)\mathscr K(x)    +     \left(\frac{q}{p}\left(\dfrac{1}{q_0}-\frac{1}{q_1}\right)-\left(\dfrac{1}{p_0}-\frac{1}{p_1}\right)\right)\kappa(x)
$$
Here $\mathscr K(\cdot)$ is the Kalton-Peck map earlier defined and $\kappa(\cdot)$ is the so-called  Kalton map \cite{kaltf}; see also \cite{cabelo}, given by $\kappa(x) = x \; r_x$ where $r_x$ is the rank function $r_x(t) = m\{s :|x(s)| >|x(t)|$ or$ |x(s)| =|x(t)|$ and $s \leq t\}$ (see \cite{ryff}).\medskip

The map $\mathscr K$ is disjointly singular while $\kappa$ has the property that every infinite dimensional subspace contains a further infinite dimensional subspace where it is trivial \cite{cabelo}, so it is clear that $\Omega$ is disjointly singular.

\item A different set of examples will be presented now in $C(K)$ or $L_\infty$ spaces. These examples are relevant because no singular quasi-linear map is possible on a space containing $\ell_1$, say $C[0,1]$ or $\ell_\infty$. It is necessary to remark that $C[0,1]$ is not a K\"othe space and thus the example lives in the domain of Banach lattices; see the comments after the examples.\end{enumerate}

\begin{proposition}\label{c0singular} There exist disjointly singular quasi-linear maps on $C[0,1]$ and $\ell_\infty$.
\end{proposition}
\begin{proof} Let us consider first the case of $C[0,1]$. As we have already remarked, one just needs to construct a $c_0$-singular map.
Let $\omega: c_0 \to C[0,1]$ be a nontrivial quasi-linear map (see \cite{ccky} for explicit examples). Let $\Gamma$ the set of all $2$-isomorphic copies $\gamma$ of $c_0$ inside $C[0,1]$, which are necessarily $4$-complemented in $C[0,1]$ via some projection $\pi_\gamma$ and
let $\alpha_\gamma: \gamma c_0$ be a 2-isomorphism. Define a quasi-linear map $\Upsilon: C[0,1]\to \ell_\infty(\Gamma, C[0,1])$ by means of
$$\Upsilon(f)(\gamma) = \omega(\alpha_\gamma \pi_\gamma(f))$$
This map is $c_0$-singular because if there is a copy of $c_0$ in which $\Upsilon$ is trivial, that copy must contains some $\gamma\in \Gamma$, on which $\Upsilon$ must be trivial too. But if $f\in \gamma$ one has
$$\Upsilon_{|\gamma}(f)(\gamma)  = \omega (\alpha_\gamma \pi_\gamma(f)) = \omega (\alpha_\gamma f)$$
thus, if $\delta_\gamma: \ell_\infty(\Gamma, C[0,1])\to C[0,1]$ is the canonical evaluation at the coordinate $\gamma$
we have obtained $\delta_\gamma \Upsilon_{|\gamma} = \omega \alpha_\gamma $. This map cannot be trivial since, otherwise, so it would be
$\omega = \delta_\gamma \Upsilon_{|\gamma} \alpha_\gamma^{-1}$, which is not. But that means that $\Upsilon_|\gamma$ cannot be trivial  because $\delta_\gamma \Upsilon_{|\gamma}$ is not trivial. \medskip

A standard reduction (see \cite{castmorestrict}) allows one to find an equivalent quasi-linear map $\Omega: C[0,1]\to \ell_\infty(\Gamma, C[0,1])$ having separable range. Since $ \ell_\infty(\Gamma, C[0,1])$ is a Banach algebra, $\Omega$ can also be considered taking values in the closed subalgebra generated by $[\Omega(C[0,1])]$, which, being separable, is contained in $C[0,1]$. Thus, $\Omega: C[0,1]\to C[0,1]$ is a $c_0$-singular quasi-linear map, as desired.\medskip

The case of $\ell_\infty$ has to be treated differently because the projections $\pi_\gamma$ do not exist now. Pick to start a
nontrivial quasi-linear map $\omega: c_0\to \ell_2$, which can be constructed as follows: pick the Kalton-Peck map $\mathscr K: \ell_2\to \ell_2$ and a quotient map $Q: C[0,1]\to \ell_2$. The map $\mathscr K Q$ is not trivial (see \cite{cabecastuni,ccky}). It cannot be locally trivial either since $\ell_2$ is reflexive and Proposition \ref{bounded} would make it trivial. Thus, using Corollary \ref{este} (cf. \cite[Theorem 2.1]{ccalb}) there is a copy of $c_0$ inside $C[0,1]$ via some isomorphic embedding $j$ so that the restriction $\mathscr K Q j$ is not trivial. Let us simplify and call this map $\omega$. Let $\Gamma$ be the set of infinite sequences of finite subsets $\N$. Given such a sequence $\gamma= (A_n)$ we will call $\overline \gamma = \cup_{A_n\in \gamma} A$. Let also $\alpha_\gamma: [1_{A_n}]\to c_0$ be an isometry. Define a quasi-linear map $\Upsilon: c_0\to \ell_\infty(\Gamma, \ell_2)$ as
$$\Upsilon(x) (\gamma) = \omega \alpha_\gamma (1_{\overline \gamma} x)$$
The bidual map $\Upsilon^{**}: \ell_\infty \to \ell_\infty(\Gamma, \ell_2)^{**}$ cannot be trivial either since
$\ell_\infty(\Gamma, \ell_2)$ is complemented in its bidual. If $\pi$ denotes a projection, the map $\pi \Upsilon^{**}: \ell_\infty \to \ell_\infty(\Gamma, \ell_2)$ cannot be trivial either. We define a new map $\Omega: \ell_\infty \to \ell_\infty(\Gamma\times \Gamma, \ell_2)$
in the form
$$\Omega (x)(\gamma, \gamma') = \pi \Upsilon^{**}(1_{\overline \gamma x})(\gamma')$$
This map $\Omega$ cannot be disjointly singular: if it becomes trivial on some $\gamma$ then
for $x\in \gamma$ one has
$$\Omega (x)(\gamma, \gamma) = \pi \Upsilon^{**}(1_{\overline \gamma x})(\gamma) = \Upsilon(x)(\gamma) = \omega \alpha_\gamma (1_{\overline \gamma x})=\omega \alpha_\gamma (x) $$

This map cannot be trivial since $\alpha_\gamma$ is an isomorphism and $\omega$ is not trivial.\end{proof}

It is an open problem posed in \cite{cabekp} whether there exists a singular
 quasi-linear map $\Omega: L_p\to L_p$ for $0 < p < 2$. Singular quasi-linear maps (not centralizers) $\Omega: L_p\to L_p$  exist for $2\leq p< +\infty$ (see \cite[Theorem 2(c)]{ccs}); observe that in this case the Kadec-Pe\l czy\'nski alternative immediately yields that a quasilinear map $\Omega$ on $L_p$  that is both disjointly singular and $\ell_2$-singular must be singular. Thus, we could use a construction similar to that in Proposition \ref{c0singular} to obtain singular maps in $L_p$, $2\leq p< +\infty$. None of these can be $L_\infty$-centralizers, nonetheless.\medskip

The papers \cite{dss1,dss2,dss3} study the behaviour of strictly singular operators in Banach lattices by considering the more general notion of \emph{lattice singular} operator (one for which no restriction to an infinite dimensional sublattice is an isomorphism). Obviously, strict singularity implies lattice singularity and this implies disjoint singularity. The authors obtain an interesting result \cite{dss1}: \emph{Let $X,Y$ be Banach lattices such that $X$ has finite cotype and $Y$  admits a lower $2$-estimate. Then an operator $T: X\to Y$ is strictly singular if and only if it is disjointly singular and $\ell_2$-singular.} A non-vacuous centralizer version for this result is not possible since
Proposition \ref{superfelix} establishes that no $L_\infty$-centralizer can be $\ell_2$-singular. It makes however sense the question of which conditions ensure that a quasi-linear map on a K\"othe space that is simultaneously disjointly singular and $\ell_2$-singular is necessarily singular.

\subsection{Characterizations}

Regarding characterizations, given a quasi-linear map $\Omega: \mathcal K \lop \mathcal K $ the fact that the twisted sum space $d_\Omega \mathcal K$ is not necessarily a K\"othe space complicates the characterization of disjointly singular maps in terms of the quotient operator. This difficulty can be overcome for centralizers, which always admit a version satisfying that $\supp \Omega x\subset \supp x$ for all $x\in \mathcal K$. Although, as we have just said, the space $d_\Omega  \mathcal K$ is not a K\"othe space, its elements are
couples of functions of $L_0$; i.e., functions $S\to \C\times \C$. The following definition makes sense:

\adef A pair of nonzero elements $f = (w_0, x_0), g=(w_1, x_1)$ of $d_\Omega \mathcal K$ are said to be disjoint if the functions $f, g: S\to \C\times \C$ are disjointly supported. An operator $\tau  \, : \,  d_\Omega \mathcal K \to \mathcal K$ is said to be disjointly singular if  the restriction of $\tau$ to any infinite dimensional subspace generated by a disjoint sequence of vectors is not an isomorphism.\zdef

One has:

\begin{lemma} \label{disjointly singular} A centralizer $\Omega$ on a maximal K\"othe space $\mathcal K$ is disjointly singular if and only if $q_\Omega$ is disjointly singular.
\end{lemma}
\begin{proof} We choose the form $q_\Omega: \mathcal K\oplus_\Omega \mathcal K\to \mathcal K $ given by $q_\Omega(v,u)= u$. If $\Omega$ is trivial on the span $[(u_n)]$ of a disjointly supported sequence
there is a linear map $L: [u_n] \to \mathcal K$ so that $\|\Omega - L\|\leq C$. From \cite[Lemma 3.17]{cfg} we get that there is a linear map $\Lambda: [u_n] \to \mathcal K$ such that $\supp \Lambda x\subset \supp x$ and $\|\Omega -\Lambda \|\leq C$ and thus $(\Lambda u_n, u_n)$ is a disjointly supported sequence on which $q_\Omega$ is trivial since
$$\|\sum \lambda_n(\Lambda u_n, u_n)\| =  \left \| \Omega\left(\sum \lambda_n u_n\right)  -  \sum \Lambda( \lambda_n u_n)  \right\| + \|\sum \lambda_n u_n\| \leq (C+1)\|\sum \lambda_n u_n\|.$$

In this way, $q_\Omega$ disjointly singular implies $\Omega$ disjointly singular. To get the converse, assume that $q_\Omega$ is not disjointly singular, so there is a disjointly supported sequence
$(v_n, u_n)$ in $\mathcal K\oplus_\Omega \mathcal K$ such that $q_\Omega$ is an isomorphism on $[(v_n, u_n)]$. This means that
$$\left \|\sum \lambda_n v_n -\Omega\left(\sum \lambda_n u_n\right) \right\| \leq C \left \|\sum \lambda_n u_n \right\| $$
The linear map $L(u_n) = v_n $ verifies
$$\left \| L\left(\sum \lambda_n u_n\right)  -  \Omega\left(\sum \lambda_n u_n\right) \right\|=\left \|\sum \lambda_n v_n  -  \Omega\left(\sum \lambda_n u_n\right) \right\|\leq C \left \|\sum \lambda_n u_n \right\|$$
\end{proof}

Now we want to mimicry Proposition \ref{sin}. Let $\mathcal K$ be a K\"othe space and let $\Omega: \mathcal K \to L_0$ be a centralizer. Given  a finite sequence $b=(b_k)_{k=1}^n \subset \mathcal K$ we will follow \cite{ccfm} and define
$$ \nabla_{[b]} \Omega   ={\rm Ave}_{\pm} \left \|  \Omega  \left ( \sn  \pm b_k\right ) - \sn \pm \Omega(b_k)\right \|,
$$
where the average is  taken over all  the signs $\pm 1$. The triangle inequality holds for $\nabla_{[b]} \Omega$: if $\Omega$ and $\Psi$ are centralizers then $\nabla_{[b]} (\Omega + \Psi) \leq \nabla_{[b]} \Omega + \nabla_{[b]} \Psi.$ If $\lambda=(\lambda_k)_k$ is a finite sequence of scalars and  $x=(x_k)_k$  a sequence of vectors of $\mathcal K$, we write $\lambda x$ to denote the finite sequence obtained by  the  non-zero vectors of $(\lambda_1x_1, \lambda_2x_2, \ldots )$.

Recall from \cite[Definiton 3.10]{cfg} that a centralizer $\Omega$ on a K\"othe function space $\mathcal K $ is \emph{contractive} if $\supp \Omega (x) \subseteq \supp x$ for every $x\in \mathcal K$.
Our next result provides a characterization of disjointly singularity for contractive centralizers in the $L_p$ spaces. The contractive restriction is not so severe, since
 every centralizer $\Omega$  on a K\"othe function space $\mathcal K $ admits a contractive centralizer $\omega$ such that $\Omega- \omega$ is bounded (\cite[Proposition 4.1]{kaltmem}).
Also it is easy to see that the canonical centralizer induced by interpolation of K\"othe spaces is contractive, see \cite{cfg}.

\begin{proposition}\label{boundedbis} A contractive centralizer  $\Omega$ defined on $L_p$ is not disjointly singular if and only if there is a disjointly supported normalized sequence $u=(u_n)_n$ and a constant $C>0$ such that for every $\lambda=(\lambda_k)_k\in c_{00}$ one has
$$   \nabla_{[\lambda u]} \Omega \leq C \| \lambda \|_p.$$
\end{proposition}

The proof follows from the following three lemmas.

\begin{lemma} Let $\Omega$ be a contractive centralizer  on a K\"othe space $\mathcal K$ satisfying an upper $p$-estimate, and let $u=(u_n)_n$ be a disjointly supported  normalized sequence of vectors. Suppose that the restriction of $\Omega$ to the closed linear span of the $u_n$'s is trivial. Then there is a constant $C>0$ such that
$$   \nabla_{[\lambda u]} \Omega \leq C  \| \lambda \|_p$$
for every $\lambda=(\lambda_k)_k\in c_{00}$.
\end{lemma}

\begin{proof}
If  the restriction of $\Omega$ to $[u_n]$  is trivial then there is a linear map $L: [u_n]\to \mathcal K$ so that $\|\Omega - L\|\leq C<+\infty$. From \cite[Lemma 3.17]{cfg} we can take such $L$ so that  $\supp L(x)\subset \supp x$.  Then for every $\lambda \in c_{00}$

$$\left \|  \Omega\left(\sum_{i=1}^n \lambda_i u_i\right) - L\left(\sum_{i=1}^n  \lambda_i u_i\right)  \right\|\leq C \left \|\sum_{i=1}^n \lambda_i u_i \right\|$$

which implies that
\begin{eqnarray*}
\left \| \Omega\left(\sum_{i=1}^n \lambda_i u_i\right)  -  \sum_{i=1}^n \Omega( \lambda_i u_i)  \right\| &=& \left \| \Omega\left(\sum_{i=1}^n \lambda_i u_i\right)  -  L(\sum_{i=1}^n \lambda_i u_i) +
\sum_{i=1}^n  \lambda_i L u_i-  \sum_{i=1}^n \lambda_i \Omega u_i)  \right\|\\
&\leq& C \left \|\sum_{i=1}^n \lambda_i u_i\right\| + \left\| \sum_{i=1}^n \lambda_i (\Omega - L) u_i  \right\|\end{eqnarray*}
From where
$$   \nabla_{[\lambda u]} \Omega \leq C'  \| \lambda \|_p$$

\end{proof}

\begin{lemma}  Let $\Omega$ be a contractive  centralizer  on a K\"othe space  $\mathcal K$. Then there exists a constant $c>0$ such that for every   disjointly supported  normalized sequence  $(v_i)$ of $\mathcal K$ and every $n\in \mathbb N$
\begin{eqnarray*} \left \| \Omega \left (\sum_{i=1}^n v_i \right) -  \sum_{i=1}^n \Omega(v_i) \right\|
&\leq& c \left \|  \sum_{i=1}^n v_i  \right\| + \nabla_{[(v_i)_{1}^n]} \Omega
 \end{eqnarray*}
\end{lemma}

\begin{proof}

 Let  $(\epsilon_i)$ be a sequence of signs. Let $v = \sum_{i=1}^n v_i$ and $\sum_{i=1}^n \epsilon_i v_i = \epsilon v$ for some function $\epsilon$ taking values $\pm 1$.

\begin{eqnarray*} \left \|\Omega(\epsilon v) - \epsilon \sum_{i=1}^n \ \Omega(v_i)\right\|
&\leq & \|\Omega(\epsilon v) - \epsilon \Omega (v)\| + \left \| \epsilon \Omega( v) -  \epsilon \sum_{i=1}^n \Omega(v_i)\right\|
\end{eqnarray*}
Thus the  centralizer $\Omega$ verifies  for some constant $c>0$,
\begin{eqnarray*} \left \| \Omega(\epsilon v) - \epsilon \sum_{i=1}^n \Omega(v_i)\right\|
&\leq& c\| v\| + \left \| \Omega (v) - \sum_{i=1}^n  \Omega(v_i)\right\|
 \end{eqnarray*}
Since $\Omega$ is contractive, then also the $\Omega(v_i)$ are disjointly supported.
Applying to  $\epsilon_i v_i$ instead of $v_i$,

\begin{eqnarray*} \left \| \Omega( v) -  \sum_{i=1}^n \Omega(v_i)\right\|
&\leq& c\| v\| + \left \| \Omega \left (\sum_{i=1}^n\epsilon_i v_i \right) - \sum_{i=1}^n  \epsilon_i \Omega(v_i)\right\|
 \end{eqnarray*}

 By taking the average
 \begin{eqnarray*} \left \| \Omega( v) -  \sum_{i=1}^n \Omega(v_i)\right\|
&\leq& c\| v\| + \nabla_{[(v_i)_i^n]} \Omega
 \end{eqnarray*}
 \end{proof}

 \begin{lemma} Let $\Omega$ be a contractive  centralizer  on a K\"othe space $\mathcal K$  satisfying a lower $q$-estimate, and let $u=(u_n)_n$ be a disjointly supported  normalized sequence of vectors. Suppose that  there is a constant $C>0$ such that
$$   \nabla_{[\lambda u]} \Omega \leq C  \| \lambda \|_q$$
for every $\lambda=(\lambda_k)_k\in c_{00}$. Then the restriction of $\Omega$ to the closed linear span of the $u_n$'s is trivial.
\end{lemma}

\begin{proof}  Let $\lambda=(\lambda_k)_k\in c_{00}$ and $(\epsilon_i)$ be a sequence of signs. It follows from the previous lemma that for some constant $c>0$ and every $n\in \mathbb N$
\begin{eqnarray*} \left \| \Omega \left (\sum_{i=1}^n \lambda_i u_i \right) -  \sum_{i=1}^n \Omega(\lambda_i u_i) \right\|
&\leq& c \left \|  \sum_{i=1}^n \lambda_i u_i  \right\| +  \nabla_{[\lambda u]} \Omega \leq c \left \|  \sum_{i=1}^n \lambda_i u_i  \right\|+ C  \| \lambda \|_q
 \end{eqnarray*}
Since $\mathcal K$  satisfies a lower $q$-estimate
\begin{eqnarray*} \left \| \Omega \left (\sum_{i=1}^n \lambda_i u_i \right) -  \sum_{i=1}^n \Omega(\lambda_i u_i) \right\|
&\leq& C' \left \|  \sum_{i=1}^n \lambda_i u_i  \right\|
 \end{eqnarray*}
 Then $\|\Omega-L\|\leq C'$, where $L$ is a  linear map such that $L(u_i)= \Omega (u_i)$.
\end{proof}

\section{Disjoint super singularity}

It is part of the folklore that ultrapowers of K\"othe spaces are again K\"othe spaces. Thus, it makes sense to define an operator $T: \mathcal K \to Y$ to be {\em super-disjointly singular} if every ultrapower of $T$ is disjointly singular; this means that for every sequence of subspaces $E_n\subseteq \mathcal K$ that are  generated by  disjointly supported elements and so that $\dim E_n=n$ there is a sequence $(F_n)$ of subspaces, $F_n\subset E_n$  generated by  disjointly supported elements such that $\dim F_n\to\infty$  and  $\lim \|T_{|F_n}\|\to 0$. To transplant these ideas to the domain of quasi-linear maps $\Omega$ on K\"othe function spaces it will be useful to define the \emph{modulus of superdisjoint singularity} of a quasi-linear map $\Omega$ as
$$ \psi_{\Omega}(n)= \inf \mathrm{dist} (\Omega|_{E_n}, L(E_n, Y)),
$$ where the infimum is taken over all  $n$-dimensional subspaces $E_n$ of $\mathcal K$ generated by disjointly supported vectors. One has:

\begin{lemma} Let $\Omega: \mathcal K \to Y$ be a quasi-linear map defined on a K\"othe space. The following are equivalent
\begin{enumerate}
\item All ultrapowers of $\Omega $ are disjointly singular.
\item $\lim \psi_\Omega(n)=+\infty$.\medskip

If $\Omega$ is a centralizer, the conditions above are equivalent to
\item  The quotient map $q_\Omega$ is super-disjointly singular.
\end{enumerate}
\end{lemma}
\begin{proof} Condition (1) says that it does not exist $c>0$ and a sequence of finite dimensional subspaces $F_n$ of $Y\oplus_\Omega \mathcal K$  such that $E_n = q_\Omega(F_n)$ is generated by $n$ disjointly supported vectors  so that $\|q_\Omega(x)\|\geq c\|x\|$ for every $x\in \cup F_n$. But if $\Omega$ is $C$-trivial on $E_n$, which is generated by the disjointly supported vectors $[u_i]_{i=1}^n$ then we claim that there is a linear map $L_n: E_n \to \mathcal K$ such that $\supp \, L_n(x) \subseteq \, \supp \, x$ for every $x\in E_n$ and $\| \Omega|_{E_n}- L_n\|\leq C$: Indeed assume  $L$ is linear such that $\|(\Omega-L)_{|E_n}\| \leq C$, and let $G$ be the finite group of units generated by the vectors $v_i$ that take value $1$ on the support of $u_i$ and $-1$ elsewhere);  then it is enough to pick $L_n(x)={\rm Ave}_{v \in G}\;\; v L(v x)$.
The rest of the argument goes as in Lemma \ref{disjointly singular}. Done that, $\Omega_\U$ is trivial on $(E_n)_\U$, which yields the equivalence between (1) and (2). The equivalence with (3) follows from Lemma \ref{disjointly singular}. \end{proof}

\adef A quasi-linear map (rep. a centralizer)  $\Omega: \mathcal K \to Y$ on a K\"othe space is said to be super disjointly singular if it satisfies the two (resp. three) equivalent conditions in the Lemma.\zdef

It is clear that either singularity or super disjoint singularity imply disjoint singularity. We will present two proofs for the following fact.

\begin{proposition} The Kalton-Peck map on $L_p$ is super disjointly singular for $1<p<\infty$.
\end{proposition}
\begin{proof} Assume there exist a linear map $L: E_n=[u] \to L_p$ and a constant $C>0$ so that $||\mathscr K_{|E_n}-L||\leq C$. By the proof of the previous lemma we may assume that $\supp \, L(u_i) \subseteq \, \supp \, u_i$ for all $i$. Put $\Omega'= \mathscr K_{|E_n}-L$ to get
$$p^{-1}n^{1/p}\log n\leq  \nabla_{[u]} \mathscr K = \nabla_{[u]} \Omega' \leq {\rm Ave}_\epsilon \| \Omega ' (\sum \epsilon_i u_i) \| + {\rm Ave}_\epsilon \| \sum \epsilon_i \Omega' u_i\|\leq 2Cn^{1/p},$$
which is impossible.\end{proof}

Two functions $f,g: \N\to \R^+$ are called equivalent, and denoted $f\sim g$,
if $0<\lim \inf f(n)/g(n)\} \leq \limsup f(n)/g(n)< +\infty$. We recall from \cite{cfg} the parameter
$$
M_{\mathcal K}(n)= \sup \{\|x_1+\ldots+x_n\| :
x_1,\ldots,x_k\; \textrm{\rm disjoint in the unit ball of}\; \mathcal K\}.
$$
The interest of this parameter lies in \cite[Proposition 5.3]{cfg}:
\begin{proposition}\label{good} Let $(X_0, X_1)$ be an interpolation couple of two K\"othe function spaces so that
$M_{X_0}$ and $M_{X_1}$ are not equivalent. Let $0<\theta<1$.
Assume that $X_\theta$ is reflexive,  that $M_W \sim M_{X_\theta}$ for every
infinite-dimensional subspace generated by a disjoint sequence $W\subset X_\theta$, and
$M_{X_{\theta}} \sim M_{X_0}^{1-\theta} M_{X_1}^\theta$.
Then $\Omega_\theta$ is disjointly singular.
\end{proposition}

We observe that:

\begin{lemma}\label{ultra} Let $X$ be a K\"othe space. Then $M_X \sim M_{X_\mathscr U}$ \end{lemma}
\begin{proof} Since $X\subset X_\U$ it is clear that $M_X\leq M_{X_\mathscr U}$. Given $n$, pick $u^k = [u^k_i]\in X_\U$ for $1\leq k\leq n$, disjointly supported  so that
$M_{X_\mathscr U}(n) \sim \|u^1 + \dots + u^n\|$ (we can freely assume that all $u^k_i$ are norm one elements). Since $u^t$ and $u^s$ are disjointly supported this means that the set of all $i$ so that $ u^s_i$ and $u^t_i$ are disjointly supported belongs to $\U$. And the same for the set $A$ of all $i$ so that all $\{u^k_i, 1\leq i\leq n\}$ are disjointly supported.
Since  $B =\{ j: \|u^1(j) + \dots + u^n(j)\|\geq  M_{X_\mathscr U}(n) - \varepsilon\} \in \U$, also $A\cap B\in \U$. Thus, for any $i\in A\cap B$ we have
$M_{X_\mathscr U}(n) - \varepsilon\ \leq \|u^1(i) + \dots + u^n(i)\|\leq M_X(n)$.
\end{proof}

In the case of K\"othe spaces, complex interpolation is actually simple factorization. Recall that given two K\"othe function spaces $Y,Z$ we define the space
$$YZ=\{ y z: y\in Y, z\in Z\}$$
endowed with the quasi-norm $\|x\|= \inf \|y\|_Y \|z\|_Z$ where the infimum is taken on all factorizations as above. Now, assuming that one of the spaces $X_0$, $X_1$ has the Radon-Nikodym property, the Lozanovskii decomposition formula allows us to show (see \cite[Theorem 4.6]{kalt-mon})
that the complex interpolation space $X_\theta$ is isometric to the space $X_0^{1-\theta} X_1^\theta$,
with
$$
\|x\|_\theta=\inf \{\|y\|_0^{1-\theta}\|z\|_1^{\theta}: y\in X_0, z\in X_1,
|x|=|y|^{1-\theta}|z|^{\theta}\}.
$$
If $a_0(x),a_1(x)$ is an $(1+\epsilon)$-optimal Lozanovskii
decomposition for $x$ then it is standard (see \cite{cfg}) that
\begin{equation}\label{kpgeneralized}
\Omega_\theta(x)  = x\, \log \frac{|a_1(x)|}{|a_0(x)|}.
\end{equation}

\begin{lemma}\label{ul} Let $(X_0,X_1)$ be a couple of Kothe function spaces with non trivial concavity. Let $\mathscr U$ be an ultrafilter on $\N$ then
$(X_\theta)_\U = ((X_0)_\U, (X_1)_\U)_\theta$ and $(\Omega_\theta)_\U = (\Omega_\U)_\theta$.
\end{lemma}
\begin{proof} According to \cite{ray}, given an interpolation couple $(A, B)$ of K\"othe spaces with non-trivial concavity their ultrapowers $(A_\U, B_\U)$ form an interpolation couple. The point now is to show that $(X_0^{1-\theta}X_1^\theta )_\U  = (X_0)_\U^{1-\theta} (X_1)_\U^{\theta}$. Indeed, given
$[x_i]\in (X_0^{1-\theta}X_1^\theta )_\U $ then pick an almost optimal factorization $x_i = y_i z_i$ and then $[x_i] = [y_i] [z_i]$ is an almost optimal factorization. Conversely, if $x \in (X_0)_\U^{1-\theta} (X_1)_\U^{\theta}$ and set $x = [y_i][z_i]$ an almost optimal factorization
then of course that $x_i = y_iz_i$ is not an almost optimal factorization for $x_i$, but it is so when the indices $i$ belong to a certain element of $\U$, and thus $[x_i]\in(X_\theta)_\U $. The assertion about the induced centralizer follows from this. \end{proof}

To apply the general criteria proved in \cite{cfg} (see below) we need to analyze the estimate $M_W$ associated to any subspace $W$ generated by a sequence of disjoint vectors of $(X_\theta)_\U$. As a rule, it is false that the ultrapower or the interpolated space is the interpolated between ultrapowers. To overcome this we concentrate first on the test case in which $X_\theta$ is an $L_p(\mu)$-space. In this situation $M_W(n) \sim n^{1/p}$ for all subspaces $W$ of $(X_\theta)_\U$ generated by disjointly supported vectors, and from Proposition \ref{good} we deduce:

\begin{proposition}\label{ppp}
Let $(X_0, X_1)$ be an interpolation couple of two K\"othe function spaces and let $0<\theta<1$ so that $X_\theta$ is an $L_p(\mu)$-space. If
\begin{equation}\label{simplesuperdisjointly singular}
\limsup \left( \left| \log \frac{M_{X_0}(n)}{M_{X_1}(n)}\right| \frac{n^{1/p}}
{M_{X_0}(n)^{1-\theta}M_{X_1}(n)^\theta}\right)=+\infty
\end{equation}
Then the induced centralizer $\Omega_\theta$
on $X_\theta$ is super disjointly singular
\end{proposition}
\begin{proof} Thanks to Lemma \ref{ultra} we observe that the hypotheses and Proposition \ref{good} imply that the centralizer $(\Omega_{\U})_\theta$ is disjointly singular. By Lemma \ref{ul} we conclude that $\Omega_\theta$ is super disjointly singular.
\end{proof}

This provides the second proof that the Kalton-Peck centralizer is super disjointly singular on $L_p$-spaces. Let us present more examples

\begin{itemize}
\item If $\mathcal S$ denotes the Schreier space then $(\mathcal S, \mathcal S^*)_{1/2}=\ell_2$ then the associated centralizer
is super disjointly singular. Since these are K\"othe sequence spaces, it is also singular. All this follows from the estimates $M_{\mathcal S}(n)=n$ and $M_{\mathcal S^*}(n) \sim \log_2(n)$.
\end{itemize}
In the case of K\"othe spaces on a discrete measure space (i.e. the unconditional basis case), as a consequence of the fact that (disjoint) singularity and "block" singularity are equivalent, the conclusion of Proposition \ref{ppp} still holds if one replaces the parameter $M_X$ by the parameter $M_X^s$, where the supremum is over sucessive vectors instead of disjointly supported. Therefore:
\begin{itemize}
\item If $\mathscr S$ denotes the Schlumprecht space then $(\mathscr S, \mathscr S^*)_{1/2}=\ell_2$ then the associated centralizer
is super disjointly singular. Since these are K\"othe sequence spaces, it is also singular. This follows from the estimates $M_{\mathscr  S}^s(n)=n$ and $M_{\mathscr S^*}^s(n) \sim \log_2(n)$.

\item Let $L_{p_0, p_1}, L_{p_1, q_1}$ be Lorentz function spaces. Then,
$(L_{p_0, p_1}, L_{p_1, q_1})_\theta = L_{p,q}$ as in Section \ref{examples} and the associated centralizer is singular when $q_0\neq q_1$ and super disjointly singular when $\min\{p_0, p_1\}\neq \min \{p_1, q_1\}$ as it follows from the estimate $M_{L_{p,q}}(n) =  n^{\frac{1}{\min\{p, q\}}}$. Observe that in this case we require a variation of Proposition \ref{ppp}: it is not true now that ``$X_\theta$ is an $L_p(\mu)$-space"; rather ``$X_\theta$ is an $L_{p,q}(\mu)$-space" and thus their subspaces generated by disjointly supported vectors are $\ell_{p,q}$, whose parameters are the same as those of $L_{p,q}$.
\end{itemize}

Although it es easy to believe that disjoint singularity implies super disjoint singularity, it is not so:

\begin{proposition} Let  $1\leq p_1< p_0\leq \infty $ and $0< \theta <1$. For $p^{-1}= (1-\theta)p_1^{-1} + \theta p_0^{-1}$ one has
$(\ell_{p_0}(\bigoplus \ell_2^n) , \ell_{p_1}(\bigoplus \ell_2^n))_{\theta}= \ell_p(\bigoplus \ell_2^k)$ with associated centralizer
$$ \Omega(x) =  \left ( \left ( \frac{p}{p_1}- \frac{p}{p_0}\right ) \log \left ( \frac{\|x^k\|_2}{\|x\|} \right ) x^k \right )_k.
$$
Thus, if $p_1<p_0<2$ then $\Omega$ is disjointly singular but not super disjointly singular
\end{proposition}
\begin{proof} The map $\Omega$ has been obtained \cite[Theorem 3.4]{new}. It is disjointly singular by Proposition \ref{good} and is not super disjointly singular since $\Omega (x)=0$ for every $x\in \ell_2^k$ and every $k\in \N$.\end{proof}

A few more precise estimates can be presented. Observe first that given a centralizer $\Omega: \mathcal K\lop \mathcal K$
on a K\"othe function space then for $n$ disjoint vectors $(u_i)_{1}^n$ in the unit ball one has
$$\left\|\Omega(\sum_{i=1}^n  u_i)-\sum_{i=1}^n \Omega(u_i)\right\|
\leq \|\Omega - L\|\left\|\sum_{i=1}^n  u_i\right\|
$$for some linear map $L: [u_n] \to \mathcal K$. Let us invoke now the estimate \cite[Proposition 5.1]{cfg}: given two K\"othe spaces $(X_0, X_1)$ (on the same base space), fixing $0<\theta<1$ and considering $\Omega_\theta$ the induced centralizer on $X_\theta= X_0^{1-\theta} X_1^{\theta}$ then
\begin{equation*}\label{eme}
\left\|\Omega_\theta\big(\sum_{i=1}^n u_i\big)-\sum_{i=1}^n \Omega_\theta(u_i)-
\log\frac{M_{X_0}(n)}{M_{X_1}(n)} \Big(\sum_{i=1}^n u_i\Big)\right\|\leq 3
\frac{M_{X_\theta}(n)}{\max\{\theta, 1-\theta\}}.
\end{equation*}

Therefore, in the same conditions as above, we get

\begin{equation*}
\left| \log \frac{M_{X_0}(n)}{M_{X_1}(n)}\right| \left\| \sum_{i=1}^n u_i \right\|\leq
\|\Omega - L\|\left\|\sum_{i=1}^n  u_i\right\| + 3 \frac{M_{X_\theta}(n)}{\max\{\theta, 1-\theta\}}.\end{equation*}

Let us define a new parameter
$$
m_{\mathcal K}(n)= \inf \{\|x_1+\ldots+x_n\| :
x_1,\ldots,x_k\; \textrm{\rm disjoint in the unit sphere of}\; \mathcal K\}.
$$
Observe that $m_{\mathcal K}(n)M_{\mathcal K^*}(n) \geq n$ for every K\"othe function space. We obtain
\begin{equation*}
\left| \log \frac{M_{X_0}(n)}{M_{X_1}(n)}\right|  m_{X_\theta}(n) \leq
\psi_{\Omega}(n) M_{X_\theta}(n) + 3 \frac{M_{X_\theta}(n)}{\max\{\theta, 1-\theta\}}.\end{equation*}
which yields
\begin{equation*}
\left| \log \frac{M_{X_0}(n)}{M_{X_1}(n)}\right| \frac{ m_{X_\theta}(n)}{ M_{X_\theta}(n)}-  \frac{3}{\max\{\theta, 1-\theta\}} \leq \psi_{\Omega_\theta}(n)
\end{equation*}

In particular:

\begin{itemize}

\item  For $\Omega_{1/2}$ the centralizer obtained on $\ell_2 = (\mathcal S, \mathcal S^*)_{1/2}$  when $\mathcal S$ is the Schreier space, one gets
$$
\left| \log n - \log \log n \right| - 6 \leq \psi_{\Omega_{1/2}}(n) $$
which shows that also the centralizers $\Omega_{1/2}$ are super-disjointly singular.

\item In general, under a few minimal conditions \cite[Proposition 6.2]{cfg}, on a K\"othe function space $\mathcal K$ with base space $S$ one has $(X, X^*)_{1/2} = L_2(S)$; and thus, if $M_{\mathcal K}$ and $M_{\mathcal K^*}$ are not equivalent then the induced centralizer on $L_2(S)$ is super disjointly singular.

\item For $\Omega_\theta$ the Kalton-Peck map on $L_p$ obtained from $X_0=\L_1, X_1= L_\infty$ and $\theta=1/p^*$ one gets
$$  \log n -  \frac{3}{\min\{\theta, 1-\theta\}} \leq \psi_{\mathscr K_p}(n)$$
which, as promised, shows again that $\mathscr K_p$ is super disjointly singular.

\item
Assume more generally that $p>1$ and $X$ is a $p$-convex K\"othe space with base space $S$. Then $X=(L_\infty(S),X^{p})_{1/p}$, where $X^{p}$ denotes the $p$-concavification of $X$, and \cite{cfg} Proposition 3.7, this induces as centralizer the map $p {\mathcal K}$, where
${\mathcal K}(f)= f \log(|f|/\|f\|)$ is the Kalton-Peck map on $X$.  Since $M_{X^p}(n)=M_X(n)^p$,
we obtain the following criteria for the super DSS property of Kalton-Peck map:
$$|\log M_X(n)| \frac{m_X(n)}{M_X(n)} -\frac{3/p}{\max{1/p,1/p'}} \leq  \psi_{\mathscr K}(n)$$

\item If for example   ${\mathcal S}^{(p)}$
is the $p$-convexification of Schreier space then since $M_{{\mathcal S}^{(p)}}(n)=M_{\mathcal S}(n)^{1/p}=n^{1/p}$ and $m_{{\mathcal S}^{(p)}}^s(n)=m_{\mathcal S}(n)^{1/p}=(n/\log n)^{1/p}$ we obtain
$$\frac{1}{p}|\log n|^{1/p'}  -\frac{3/p}{\max{1/p,1/p'}} \leq  \psi_{\mathscr K}(n)$$
and deduce that Kalton-Peck map is super disjointly singular on ${\mathscr S}^{(p)}$.

\item The same estimates hold, in the case of K\"othe sequence spaces, using the successive vectors versions $M_X^s(n)$, $m_X^s(n)$ and $\psi^s_X(n)$ of the parameters and of the modulus. So for example, if  ${\mathscr S}^{(p)}$
is the $p$-convexification of Schlumprecht space then since $M^s_{{\mathscr S}^{(p)}}(n)=n^{1/p}$ and $m_{{\mathscr S}^{(p)}}^s(n)=(n/\log n)^{1/p}$ we also obtain
$$\frac{1}{p}|\log n|^{1/p'}  -\frac{3/p}{\max{1/p,1/p'}} \leq  \psi^s_{\mathscr K}(n)$$
and deduce that Kalton-Peck map is "super successively singular", therefore disjointly singular, hence singular on ${\mathscr S}^{(p)}$.
\end{itemize}

\

{\em Acknowledgements:} The authors thank F\'elix Cabello and  Pedro Tradacete for fruitful discussions about the topics in the paper.

.

\end{document}